\documentclass{amsart}

\DeclareMathOperator{\ran}{ran}
\DeclareMathOperator{\vecc}{vec}
\newcommand{\dfn}[1]{{\bf #1}\index{#1}}
\newcommand{\bbm}{\left[ \begin{matrix}}
\newcommand{\ebm}{\end{matrix} \right]}
\newcommand{\bpm}{\left( \begin{smallmatrix}}
\newcommand{\epm}{\end{smallmatrix} \right)}
\newcommand\beq{\begin{equation}}
\newcommand\eeq{\end{equation}}

\newcommand{\inv}{^{-1}}
\newcommand{\cc}[1]{\overline{#1}}
\theoremstyle{theorem}
\newtheorem{theorem}{Theorem}[section]
\newtheorem{lemma}{Lemma}[section]
\newtheorem{proposition}{Proposition}[section]
 
\theoremstyle{definition}

\begin{document}

\title[Computing matrix sub-algebras]{An elementary method to compute the algebra generated by some given matrices and its dimension}
\markright{Computing matrix sub-algebras}
\author{
J. E. Pascoe$^\dagger$
}
\address{Department of Mathematics\\
 University of Florida\\
  1400 Stadium Rd. \\
 Gainesville,  FL 32611}
\email[J. E. Pascoe]{pascoej@ufl.edu}
\thanks{$\dagger$ Partially supported by National Science Foundation Mathematical
Science Postdoctoral Research Fellowship  
DMS 1606260}

\maketitle

\begin{abstract}
	We give an efficient solution to the following problem:
	Given $X_1, \ldots X_d$ and $Y$ some $n$ by $n$ matrices
		can we determine if $Y$ is in
	the unital algebra generated by $X_1, \ldots, X_d$
	as a subalgebra of all $n$ by $n$ matrices?
	The solution also gives an easy method for computing the dimension of this algebra.
\end{abstract}

	\section{The problem}
	%Let $X = (X_1, \ldots X_d)$ be a tuple of matrices in $M_n(\mathbb{C}).$
	%We define the \dfn{algebra generated by $X,$} denoted $\mathcal{A}_X,$
	%to be the unital subalgebra of $M_n(\mathbb{C})$ generated by
	%$X_1, \ldots, X_d.$
	
	For example, given
		%$$ \label{eqnone} X_1 = 
        %\bbm \frac{1}{3} & 0  \\ 0 & 0 \ebm,
		%X_2 = \bbm 0 & \frac{1}{3} \\ 0 & 0 \ebm, $$
		$$
			X_1 = \bbm
				\frac{1}{3} & 0 & 0 \\
				0 & 0 & 0 \\
				0 & 0 & 0
			\ebm, 
			X_2 =\bbm
				0 & \frac{1}{3} & 0 \\
				0 & 0 & \frac{1}{3} \\
				0 & 0 & 0
			\ebm
		$$
	and 
		%$$ \label{eqntwo} Y = \bbm \frac{2}{5} & \frac{1}{5} \\ 0 & \frac{1}{5} \ebm, $$
		$$ \label{eqntwo} Y = \bbm
			1 & 0 & 1 \\
			0 & 1 & -1 \\
			0 & 0 & 1
			\ebm
		$$
	is $Y$ in the unital algebra generated by $X_1$ and $X_2?$
	That is, we want to know whether $Y$ is in the span of all words in $X_1$ and $X_2.$
	Indeed, it is, for example
	$$ Y = 1 - 3X_2 + 9 X_1X_2 + 9X_2^2 $$
	%(The representation for $Y$ is highly non-unique. In fact, in this example, $Y$ happens to be in the linear span of $X_1, X_2$ and the identity. We have intentionally chosen an artificially complicated representation to evoke the general situation, as for higher dimensional problems one may need many somewhat complicated words in the generators to express a given element of an algebra.)
	However, if we had instead chosen
		$$\hat{Y} =\bbm
			1 & 0 & 1 \\
			0 & 1 & -1 \\
			1 & 0 & 1
			\ebm,$$
	we see that $\hat{Y}$ fails to be in the algebra generated by $X_1$ and $X_2$ for obvious reasons having to do with $X_1$ and $X_2$ being
upper triangular.

Some computer algebra systems currently possess functionality to do this calculation, such as GAP and Magma. As of December 2018, in \textsc{gap-4.8,} a basis of such an algebra is computed from given algebra generators
by forming products and using Gaussian elimination at each step to see if the new product was already in the span of the currently generated basis without additional sophistication \cite{GAPPRIVATE, GAP4}.  Magma is proprietary software, and their engineers could not be reached for comment.
We will give a method to calculate the algebra and its dimension that can take advantage of fast algorithms for matrix multiplication and inversion, essentially by calculating the entire basis at once.

\section{Some rearrangements of matrices and the Kronecker product}
We will need several important operations on matrices which we will now define. Most of these will be familiar, except the $\psi$-involution, which is ostensibly new.
%The \dfn{transpose} of an $n$ by $m$ matrix $A$, denoted $A^T,$ is an $m$ by $n$ matrix formed by rearranging the entries
%so that the $(i,j)$ entry of $A$ is the $(j,i)$ entry of $A^T.$
%For example,
%$$
%\bbm
%e & \pi \\ 
%\gamma & \phi
%\ebm^T = \bbm
%e & \gamma \\ 
%\pi & \phi
%\ebm
%.$$
%A matrix is called \dfn{symmetric} if $A^T = A.$
The \dfn{vectorization} of an $n$ by $m$ matrix $A,$ denoted $\vecc A$ rearranges the matrix $A$ into a column vector
by stacking each of the columns on top of each other. Specifically, the $(i,j)$ entry of $A$ becomes the $i + j(n-1)$-th coordinate
of $\vecc A.$
For example,
$$\vecc \bbm
1 & 3 \\ 
2 & 4
\ebm = \bbm
1 \\
2 \\
3\\
4
\ebm$$
The \dfn{$\psi$ involution} of an $nm$ by $pq$ matrix $A,$ denoted $A^{\psi},$
rearranges $A$ into an $np$ by $mq$ matrix so that the $(i+(j-1)n, k+(l-1)p).$
entry of $A$ becomes the $(i+(k-1)p,j+(l-1)m)$ entry of $A^{\psi},$
where $i$ ranges from $1$ to $n,$ $j$ ranges from $1$ to $m,$ $k$ ranges from $1$ to $p,$ and
$l$ ranges from $1$ to $q.$
(Note that the definition of $\psi$ depends on $n, m, p,$ and $q,$ and not just $nm$ and $pq$ themselves.
For the purposes of this discussion, we will always have $n=m=p=q.$)
For example when $n=m=p=q=2,$
$$
\bbm
1 & 5 & 9  & 13 \\
2 & 6 & 10 & 14 \\
3 & 7 & 11 & 15 \\
4 & 8 & 12 & 16
\ebm^\psi =
\bbm
1 &3 &9 &11\\
2 &4 &10&12\\
5 &7 &13&15\\
6 &8 &14&16
\ebm
.$$
Note that $(A^\psi)^\psi = A,$ since $\psi$ switches the roles of $j$ and $k,$ and repeating the operation switches them back, so $\psi$ is indeed an involution.
A perhaps better way to understand the $\psi$ involution, which is evident from the example, is to view the 
 $nm$ by $pq$ matrix $A$ as a block $m$ by $q$ matrix with entries that are themselves $n$ by $p$ matrices, that is,
$$
A = 
\bbm
A_{11} & \cdots & A_{1q}  \\
\vdots & \ddots & \vdots \\
A{m1} & \cdots & A_{mq} \\
\ebm
,$$
 and observe that
$A^{\psi}$ lists out the vectorizations of these block entries:
\beq \label{vecpsi} A^\psi = \bbm
A_{11} & \cdots & A_{1q}  \\
\vdots & \ddots & \vdots \\
A{m1} & \cdots & A_{mq} \\
\ebm^\psi = \bbm \vecc A_{11} & \vecc A_{21} & \ldots & \vecc A_{mq} \ebm.\eeq
The \dfn{Kronecker product} of a $n$ by $m$ matrix $A$ and an $p$ by $q$ matrix $B,$ denoted $A \otimes B$
is an $np$ by $mq$ matrix such that the $(i+(k-1)n,j+(l-1)m)$ entry is given by the $(i,j)$ entry of $A$
times the $(k,l)$ entry of $B$ where $i$ ranges from $1$ to $n,$ $j$ ranges from $1$ to $m,$ $k$ ranges from $1$ to $p,$ and
$l$ ranges from $1$ to $q.$
%{\red FIX, backward?}
%For example,
%$$\bbm 1 & 2 \\
%3 & 5
%\ebm \otimes \bbm 7 & 11 \\
%13 & 17
%\ebm
%= \bbm
%7 & 14& 11& 22 \\
%21& 35& 33& 55 \\
%13& 26& 17& 34 \\
%39& 65& 51& 85
%\ebm.$$
A more convienient formulation comes by viewing $A \otimes B$ as a block $p$ by $q$ matrix where each block $n$ by $m$ entry is given by $b_{ij}A,$ that is,
$$A \otimes B =
\bbm
b_{11}A & \cdots & b_{1q}A\\
\vdots & \ddots & \vdots \\
b_{p1}A & \cdots & b_{pq}A
\ebm.
$$
The Kronecker product has the following relation:
$$(A\otimes B)(C \otimes D) = AC \otimes BD.$$
There is an important interaction between the maps given by the following Proposition.
\begin{proposition}\label{impeq}
$$ \left(A \otimes {B}\right)^{\psi} = (\vecc A)(\vecc {B})^T.$$
\end{proposition}
\begin{proof}
Note
	$$A \otimes B =
\bbm
b_{11}A & \cdots & b_{1q}A\\
\vdots & \ddots & \vdots \\
b_{p1}A & \cdots & b_{pq}A
\ebm.
$$
So, 
\begin{align*}\left(A \otimes {B}\right)^{\psi} & =
\bbm
b_{11}A & \cdots & b_{1q}A\\
\vdots & \ddots & \vdots \\
b_{p1}A & \cdots & b_{pq}A
\ebm^\psi \\
& =
\bbm \vecc b_{11}A & \vecc b_{21}A & \ldots & \vecc b_{pq}A \ebm & \textrm{ by Equation \eqref{vecpsi}}\\
& =
\bbm b_{11}\vecc A & b_{21}\vecc A & \ldots & b_{pq}\vecc A \ebm\\
& =
(\vecc A)(\vecc {B})^T.
\end{align*}
\end{proof}
%which is left as an important exercise to the reader. (Hint: This is easier if you use the block matrix characterizations of
%$\psi$ and the Kronecker product.)

\section{The main result}
Now thoroughly equipped, we state our result. We note that we will require a norm bound on the data, so in general to solve the problem, one may have to rescale.
	\begin{theorem}\label{mainresult}
%Let $\mathcal A$ be the unital algebra generated by $X_1, \ldots, X_d$
%such that
Let $X_1, \ldots, X_d$ be $n$ by $n$ matrices over $\mathbb{R}$ such that
$\|\sum_i X_i \otimes X_i\|<1,$ where $\|\cdot\|$ is any consistent matrix norm on $n^2$ by $n^2$ matrices over $\mathbb{R}.$
	Let
	$$P = \left[\left(1 - \sum_i  X_i \otimes X_i\right)\inv\right]^\psi.$$
	The matrix $P$ is symmetric and positive semi-definite and:
	\begin{enumerate}
		\item $Z$ is in the unital algebra generated by $X_1, \ldots, X_d$ exactly when $\vecc Z \in \ran P,$
		\item The dimension of the unital algebra generated by $X_1, \ldots, X_d$ is equal to the rank of the matrix $P.$
	\end{enumerate}
\end{theorem}
Note that $\|\sum_i X_i \otimes X_i\|<1$ for some consistent matrix norm if and only if the spectral radius of $\sum_i X_i \otimes X_i$ is less than $1.$
In practice, some consistent matrix norms are easier to compute than others and in particular much easier than the spectral radius. For example, the $\ell^1, \ell^\infty$ and Frobenius norm
are very easy to calculate. We discuss several modified versions of $P$ in Section \ref{pomr} after the proof of Theorem \ref{mainresult} which handle other cases, such as algebras over $\mathbb{C},$ the nonunital case, and a (slower to evaluate) version which does not require a norm bound.

\subsection{Example}
Before we formally prove the theorem, let us attempt an example to see what it really does for us.
We will now apply our technique given in Theorem \ref{mainresult} to the example from the introduction.
Again, take
%$$ X_1 = \bbm \frac{1}{3} & 0 \\ 0 & 0 \ebm,
%		X_2 = \bbm 0 & \frac{1}{3} \\ 0 & 0 \ebm.$$
$$
			X_1 = \bbm
				\frac{1}{3} & 0 & 0 \\
				0 & 0 & 0 \\
				0 & 0 & 0
			\ebm, 
			X_2 =\bbm
				0 & \frac{1}{3} & 0 \\
				0 & 0 & \frac{1}{3} \\
				0 & 0 & 0
			\ebm
		$$
		
First, let us calculate $X_1 \otimes X_1 + X_2 \otimes X_2,$
	\begin{align*}
		X_1 \otimes X_1 + X_2 \otimes X_2 &= \bbm \frac{1}{3} & 0 & 0 \\
				0 & 0 & 0 \\
				0 & 0 & 0
			\ebm \otimes \bbm
			\frac{1}{3} & 0 & 0 \\
				0 & 0 & 0 \\
				0 & 0 & 0
			\ebm
			+
			\bbm
				0 & \frac{1}{3} & 0 \\
				0 & 0 & \frac{1}{3} \\
				0 & 0 & 0
			\ebm
			\otimes
			\bbm
				0 & \frac{1}{3} & 0 \\
				0 & 0 & \frac{1}{3} \\
				0 & 0 & 0
			\ebm \\
			&=
			\bbm
			\frac{1}{9} & 0 & 0 & 0 & 0 & 0 & 0 & 0 & 0 \\
			0 & 0 & 0 & 0 & 0 & 0 & 0 & 0 & 0 \\
			0 & 0 & 0 & 0 & 0 & 0 & 0 & 0 & 0 \\
			0 & 0 & 0 & 0 & 0 & 0 & 0 & 0 & 0 \\
			0 & 0 & 0 & 0 & 0 & 0 & 0 & 0 & 0 \\
			0 & 0 & 0 & 0 & 0 & 0 & 0 & 0 & 0 \\
			0 & 0 & 0 & 0 & 0 & 0 & 0 & 0 & 0 \\
			0 & 0 & 0 & 0 & 0 & 0 & 0 & 0 & 0 \\
			0 & 0 & 0 & 0 & 0 & 0 & 0 & 0 & 0
			\ebm
			+
			\bbm
			0 & 0 & 0 & 0 & \frac{1}{9} & 0 & 0 & 0 & 0 \\
			0 & 0 & 0 & 0 & 0 & \frac{1}{9} & 0 & 0 & 0 \\
			0 & 0 & 0 & 0 & 0 & 0 & 0 & 0 & 0 \\
			0 & 0 & 0 & 0 & 0 & 0 & 0 & \frac{1}{9} & 0 \\
			0 & 0 & 0 & 0 & 0 & 0 & 0 & 0 & \frac{1}{9} \\
			0 & 0 & 0 & 0 & 0 & 0 & 0 & 0 & 0 \\
			0 & 0 & 0 & 0 & 0 & 0 & 0 & 0 & 0 \\
			0 & 0 & 0 & 0 & 0 & 0 & 0 & 0 & 0 \\
			0 & 0 & 0 & 0 & 0 & 0 & 0 & 0 & 0
			\ebm\\
			&=
			\bbm
			\frac{1}{9} & 0 & 0 & 0 & \frac{1}{9} & 0 & 0 & 0 & 0 \\
			0 & 0 & 0 & 0 & 0 & \frac{1}{9} & 0 & 0 & 0 \\
			0 & 0 & 0 & 0 & 0 & 0 & 0 & 0 & 0 \\
			0 & 0 & 0 & 0 & 0 & 0 & 0 & \frac{1}{9} & 0 \\
			0 & 0 & 0 & 0 & 0 & 0 & 0 & 0 & \frac{1}{9} \\
			0 & 0 & 0 & 0 & 0 & 0 & 0 & 0 & 0 \\
			0 & 0 & 0 & 0 & 0 & 0 & 0 & 0 & 0 \\
			0 & 0 & 0 & 0 & 0 & 0 & 0 & 0 & 0 \\
			0 & 0 & 0 & 0 & 0 & 0 & 0 & 0 & 0
			\ebm
	\end{align*}
The Frobenius norm of the above matrix is $\sqrt{\frac{5}{81}}<1,$ so we may apply Theorem \ref{mainresult}.
Now,
\begin{align*}
	P = & \left[\left(1 - X_1 \otimes X_1 - X_2 \otimes X_2\right)\inv\right]^\psi \\
	= & \left[\bbm
\frac{8}{9} & 0 & 0 & 0 & -\frac{1}{9} & 0 & 0 & 0 & 0 \\
0 & 1 & 0 & 0 & 0 & -\frac{1}{9} & 0 & 0 & 0 \\
0 & 0 & 1 & 0 & 0 & 0 & 0 & 0 & 0 \\
0 & 0 & 0 & 1 & 0 & 0 & 0 & -\frac{1}{9} & 0 \\
0 & 0 & 0 & 0 & 1 & 0 & 0 & 0 & -\frac{1}{9} \\
0 & 0 & 0 & 0 & 0 & 1 & 0 & 0 & 0 \\
0 & 0 & 0 & 0 & 0 & 0 & 1 & 0 & 0 \\
0 & 0 & 0 & 0 & 0 & 0 & 0 & 1 & 0 \\
0 & 0 & 0 & 0 & 0 & 0 & 0 & 0 & 1
\ebm
\inv\right]^\psi \\
= & 
\bbm
\frac{9}{8} & 0 & 0 & 0 & \frac{1}{8} & 0 & 0 & 0 & \frac{1}{72} \\
0 & 1 & 0 & 0 & 0 & \frac{1}{9} & 0 & 0 & 0 \\
0 & 0 & 1 & 0 & 0 & 0 & 0 & 0 & 0 \\
0 & 0 & 0 & 1 & 0 & 0 & 0 & \frac{1}{9} & 0 \\
0 & 0 & 0 & 0 & 1 & 0 & 0 & 0 & \frac{1}{9} \\
0 & 0 & 0 & 0 & 0 & 1 & 0 & 0 & 0 \\
0 & 0 & 0 & 0 & 0 & 0 & 1 & 0 & 0 \\
0 & 0 & 0 & 0 & 0 & 0 & 0 & 1 & 0 \\
0 & 0 & 0 & 0 & 0 & 0 & 0 & 0 & 1
\ebm
^\psi \\
= & 
\bbm
\frac{9}{8} & 0 & 0 & 0 & 1 & 0 & 0 & 0 & 1 \\
0 & 0 & 0 & 0 & 0 & 0 & 0 & 0 & 0 \\
0 & 0 & 0 & 0 & 0 & 0 & 0 & 0 & 0 \\
0 & 0 & 0 & \frac{1}{8} & 0 & 0 & 0 & \frac{1}{9} & 0 \\
1 & 0 & 0 & 0 & 1 & 0 & 0 & 0 & 1 \\
0 & 0 & 0 & 0 & 0 & 0 & 0 & 0 & 0 \\
0 & 0 & 0 & 0 & 0 & 0 & \frac{1}{72} & 0 & 0 \\
0 & 0 & 0 & \frac{1}{9} & 0 & 0 & 0 & \frac{1}{9} & 0 \\
1 & 0 & 0 & 0 & 1 & 0 & 0 & 0 & 1
\ebm.
\end{align*}
Note the vectorization of any matrix of the form $\bbm a & b & c \\ 0 & d & e \\ 0&0& d \ebm$ is in the range of $P,$ and that the rank of $P$ is $5.$
That is, we know any upper triangular matrix such that the $(2,2)$ and $(3,3)$ entry are equal, such as 
$$ Y = \bbm
			1 & 0 & 1 \\
			0 & 1 & -1 \\
			0 & 0 & 1
			\ebm  $$
is in the algebra generated by $X_1$ and $X_2$ and that the dimension of the unital algebra generated by $X_1$ and $X_2$ is exactly $5$
by Theorem \ref{mainresult}.

\subsection{Proof of Theorem \ref{mainresult}}\label{pomr}
Before we can prove the main result, we need the following lemma.
\begin{lemma}\label{triviallemma}
Let $(v_i)^\infty_{i=1}$ be a sequence of vectors in $\mathbb{R}^n$ such that $\sum^\infty_{i=1} v_iv_i^T$ converges.
The matrix $\sum^\infty_{i=1} v_iv_i^T$ is symmetric and positive semi-definite and has range exactly equal to the span of the $v_i,$ and has kernel perpendicular to the range.
\end{lemma}
\begin{proof}
	Let $P=\sum^\infty_{i=1} v_iv_i^T.$
	The set of symmetric positive semi-definite matrices is a closed cone and each $v_iv_i^T$ is symmetric and positive semi-definite. Therefore, the same must be true of $P.$
	
	Suppose $w$ is in the kernel of $P,$ that is $Pw = 0.$ Therefore, $w^TPw= 0.$
	Note that $$w^TPw = w^T \sum^\infty_{i=1} v_iv_i^T w = \sum^\infty_{i=1} |\langle w, v_i\rangle|^2.$$
	Thus, $\langle w, v_i\rangle = 0$ for all $i,$ which says that $w$ is perpendicular to $v_i.$
	Finally, the range of a symmetric real matrix must be perpendicular to its kernel, so we see that each $v_i$ must be in the range of $P.$
\end{proof}

We now prove the main result.
\begin{proof}
%To begin, we note that
%the map $\psi$ satisfies the following equation for any $A, B\in M_n(\mathbb{C})$
%	$$\left(\overline{A} \otimes {B}\right)^{\psi} = (\vecc B)(\vecc {A})^*.$$
%Explicitly, checking on a basis,
%\begin{align*}
%(E_{i,k} \otimes E_{j,l})^\psi & = [E_{i+(j-1)n,k+(l-1)n}]^\psi \\
%&  =  E_{i+(k-1)n, j+(l-1)n} \\
%& = (\vecc E_{j,l})(\vecc E_{i,k})^*.
%\end{align*}

Since $\|\sum_i X_i \otimes X_i\|<1,$  one can expand $\left(1 - \sum_i X_i \otimes X_i\right)\inv$
as a geometric series.
So now,
	\begin{align*}
	P&=	\left[\left(1 - \sum_i  X_i \otimes X_i\right)\inv\right]^\psi \\
	&= \left[\sum^{\infty}_{k=0}\left(\sum_i  X_i \otimes X_i\right)^k\right]^{\psi} \\
	&= \sum^{\infty}_{k=0} \left[\left(\sum_i  X_i \otimes X_i\right)^k\right]^{\psi} \\
	& = \sum^{\infty}_{k=0} \sum_{i_1,\ldots,i_k} \left[({X_{i_1}\ldots X_{i_k}}) \otimes 
	(X_{i_1}\ldots X_{i_k})\right]^{\psi}\\
	& = \sum^{\infty}_{k=0} \sum_{i_1,\ldots,i_k} (\vecc X_{i_1}\ldots X_{i_k})
	(\vecc X_{i_1}\ldots X_{i_k})^T  \textrm{ by Proposition \ref{impeq}.}
	\end{align*}	
	
So the range of $P$ is exactly the span of $\vecc X_{i_1}\ldots X_{i_k}$ over all words by Lemma \ref{triviallemma}, so we are done.
\end{proof}

Note that, in principle, one can use fast algorithms (Strassen, Coppersmith-Winograd, parallel computing, etc.) for matrix inversion to compute the special matrix $P$ in Theorem
\ref{mainresult}. It is important to note that if the dimension of the algebra is less than $n^2,$ then $X_1, \ldots, X_d$
must have a nontrivial joint invariant subspace (over $\mathbb{C}$) by Burnside's theorem, and therefore the theorem gives an easy way to compute if such a subspace exists,
although it is unclear how to find the subspace itself from $P.$ 
If one wants to consider matrices over $\mathbb{C},$ one can replace $P$ by 
$$\left[\left(1 - \sum_i  X_i \otimes \cc X_i\right)\inv\right]^\psi$$
with essentially the same proof with transpose replaced by adjoint. (Here, $\cc A$ is the matrix $A$ with all the entries complex conjugated.)
Additionally, if one wants to consider the non-unital algebra, one can replace $P$ by 
$$\left[\sum_i \left( X_i \otimes \cc X_i\right)\left(1 - \sum_i  X_i \otimes \cc X_i\right)\inv\right]^\psi.$$
Finally, we note that the method can be adapted to compute the intersection of an algebra generated by $X_1, \ldots, X_d$
and $\tilde{X}_1, \ldots, \tilde{X}_{\tilde{d}}$ by computing the intersection of the ranges of the corresponding $P$ and $\tilde{P}.$ The proof is essentially the same as Theorem \ref{mainresult}. One should probably imagine any sufficiently nice and analytic expression in terms of $X_i \otimes X_i$ gives some kind of \emph{spatial generating function}.

We caution that in the case where we lack the norm bound or of fields with positive characteristic, it is unclear that $P$ tells us anything even when it exists. The formulas in the proof of Theorem \ref{mainresult} cease to make sense, or could have cancellation if you choose to expand the geometric series about a different point.

Over the complex numbers, however, one can rectify the need for a norm bound by replacing $P$ by the matrix 
$$\left[\left(1 + \sum_i  X_i \otimes \cc X_i\right)^{k}\right]^\psi$$
where $k$ is large enough so that the words of degree $k$ must generate the algebra. It is clear that the choice of $k = n^2$ works.
If we believe the Paz conjecture \cite{Paz84}, one could take $k = 2n-2,$ although the best known bounds, obtained very recently by Shitov in \cite{Shitov}, give that we can choose $k= 2n \log_2 n + 4n,$ improving the best previously known bounds by Pappacena \cite{Pap97}.
All such methods necessitate a slowdown required to evalute such an exponentiation. 
%Over numerical arithmetic using successive squaring would require a logarithmic factor slowdown and over 
%rational arithmetic much more due to the integer entries of such a matrix becoming quite large.

Although the problem is probably of general interest, we were motivated to solve this problem because of the work of Agler and McCarthy in
	\cite{aglermcnpfree}, where they used algebra membership as a hypothesis in their 
	solution to the matricial noncommutative analogue of Nevanlinna-Pick interpolation. They regarded the
	problem of algebra membership as delicate, so we decided to give the above explicit solution. We also note that prior works of O'Meara \cite{GerstenhaberDec} and Holbrook-O'Meara \cite{GerstenhaberLAA}  give
	some reasonably efficient methods in the commutative case in relation to the Gerstenhaber problem: what are the possible dimensions of a $3$ generated commutative algebra of $n$ dimensional matrices.

In closing, we further note that numerical implementations of Theorem \ref{mainresult} have shown that
$P$ often has small eigenvalues, so sometimes numerically it looks like the rank of $P$ is much lower than
it actually should be according to the theorem.
(Generically, the dimension of algebra generated by $2$ or more $n$ by $n$ matrices is equal to $n^2.$) The reason why is that often times the quantity
$\left(\sum_i  X_i \otimes X_i\right)^k$ can rapidly go to $0.$ In fact, often the matrix $P$
had bands of eigenvalues of size $d^n$ on random inputs.

\section{Computing the dimension for integer matrices}
Of particular interest in the algebra dimension problem are matrices over $\mathbb{Z},$ and hence over $\mathbb{Q}$ by clearing denominators. This section gives an implementation of our method that shows that for a generic prime $p$, one can do the computation in Theorem \ref{mainresult} modulo $p.$

We define the \dfn{Frobenius norm} to be $$\|M\|_F = \sqrt{\text{tr}(MM^T)} = \sqrt{\sum_{i,j} m_{ij}^2}.$$
%Note that $$\|M\otimes N\|_F = \|M\|_F\|N\|_F.$$
The Frobenius norm gives an important bound on the determinant for $n$ by $n$ matrices, the \dfn{volume bound}:
    $$|\det(M)| \leq \left(\frac{\|M\|_F}{\sqrt{n}} \right)^n,$$
which holds because the unsigned determinant measures the volume of the parallelepiped cut out by the columns of $M$ and the worst case is a cube.

\begin{theorem} \label{theoremprime}
Let $X_1, \ldots, X_d$ be $n$ by $n$ matrices over $\mathbb{Z}.$
%Consider $T = \sum X_i \otimes X_i.$
Take $B= \lceil{\sum \|X_i\|_F^2}\rceil+1.$
Let
$$P = \left[\left(B - \sum X_i \otimes X_i\right)^{-1}\right]^\psi.$$
For all but at most $n^2(n^2+1)\log B + n(n^2+1) + n^2\log n$ primes,  $P \pmod{p}$ is well-defined and given by the algebraic expression for $P$ evaluated modulo $p,$ and the rank of $P \pmod{p}$ is equal to the dimension of the algebra generated by $X_1, \ldots, X_d.$
\end{theorem}
\begin{proof}
Note the tuple of  $X_i / \sqrt{B}$ satisfy the hypotheses of Theorem \ref{mainresult}.
So after some rescaling, we need to compute the rank of the integer matrix
\begin{align*}
    \hat{P} & = \det{\left(B - \sum X_i \otimes X_i\right)} P \\
    &= \det{\left(B - \sum X_i \otimes X_i\right)}\left[\left(B - \sum X_i \otimes X_i\right)^{-1}\right]^\psi \\
    &= \left[\text{adj}\left(B - \sum X_i \otimes X_i\right)\right]^\psi .
\end{align*}
(The extra factor $\det{\left(B - \sum X_i \otimes X_i\right)}$ in the formula for $\hat{P}$ will be immaterial to resulting set of nice $p$ as we will exclude its factors.)
Each of the entries of $\hat{P}$ must be less than $\left(\frac{\|\left(B - \sum X_i \otimes X_i\right)\|_F}{n}\right)^{n^2}$ by Cramer's rule combined with the volume bound on the determinant.
Simplifying, using the observation that $\|B - \sum X_i \otimes X_i\|_F \leq (n+1)B$, we can see that the entries must be less than $e^nB^{n^2}.$ Moreover, we know that 
$\det{\left(B - \sum X_i \otimes X_i\right)}$ satisfies the same bound.
Therefore, the determinant of any minor $M$ of $\hat{P}$ 
must be less than 
$n^{n^2}e^{n^3}B^{n^4}$ by applying the volume bound.

Note that if a prime $p$ does not divide $\det M \det \hat{P}$ for a minor witnessing the rank of $\hat{P},$ then we could have evaluated the formula for $P$ modulo $p$ and obtained a matrix with the same rank as $P.$
(In general, the rank of $P \pmod{p}$ must be less than the rank of $P.$) We know that $\det M \det \hat{P} \leq n^{n^2}e^{n^3+n}B^{n^4+n^2}$ by previous estimates. A number $L$
can have at most $\min\{\log L,2\}$ distinct prime factors. Applying this observation and our estimate, we see that $n^2(n^2+1)\log B + n(n^2+1) + n^2\log n$ possible primes for which $P$ might not give us the rank.
%First, observe that the rank of $P$ modulo a prime $p$ must be less than or equal to
%the rank of $P$, and therefore we need to compute the number of exceptional primes such that the rank goes down. Secondly, there is a minor $M$ which witnesses the rank of $P.$ That is, $M$ is a minor with rank equal to that of $P$ and non-zero determinant. We want to count how many primes might divide the determinant of such a minor. 
%Using the previous bound for the unsigned determinant of such a minor, we see that 
%there are at most $n^4\log B+ n^3 + n^2\log n$ exceptional primes, as such a number can only have at most that many distinct factors.
%Adding in the potential $n^2\log B + n$ factors of $\det{\left(B - \sum X_i \otimes X_i\right)},$ we see that there are at most $n^2(n^2+1)\log B + n(n^2+1) + n^2\log n$ possible primes for which $P$ might not give us the rank.
\end{proof}

The number of primes below $N$ is about $\frac{N}{\log N}$ by the prime number theorem. So picking a random prime below $N$ has a less than
$\frac{(n^2(n^2+1)\log B + n(n^2+1) + n^2\log n)\log N}{N}$ probability of giving the wrong rank in the above theorem.
One can then check the dimension by taking a random prime on the order of
$n^4\log(B)(\log n + \log \log B)$ and computing the rank. This works most of the time, but one can pick several if additional certainty is required.
If we believe that the determinant of a minor is truly a random unstructured number, it is likely to have on the order of $\log(n^2(n^2+1)\log B + n(n^2+1) + n^2\log n)$ prime factors, as classical results  of Hardy and Ramanujan \cite{HARDY} state that a number $N$ is expected to have about  $\log \log N$ factors. In an average situation, we expect the probability of failure for a particular prime to be much lower than the theoretical guarantee from Theorem \ref{theoremprime}. Furthermore, under such an unfounded unstructured assumption, we would expect that each prime divides the determinant with probability about $\frac{1}{p}$ which means heuristically, one can test only a few primes and obtain reasonable certainty independent of $n$.

\section{Acknowledgements}
In early 2017, an implementation of the method was made in Mathematica by Igor Klep to find a counterexample to an unpublished conjecture of Knese about extreme points of the cone of rational inner Herglotz functions, as in acknowledged in Section 10 of \cite{knese}. The Mathematica code has been used modified and shared for other problems by Klep and others. The author would like to thank Klep for the implementation and encouragement to submit this manuscript. We thank Kevin O'Meara for some helpful comments on the relation to other problems, and encouragment to give explicit bounds for probabilistic methods in the integral case. Finally, we would like to thank the thoughtful referee for several helpful suggestions.

%\begin{acknowledgment}{Acknowledgment.}
%The authors wish to thank the Greek polymath Anonymous, whose prolific works are an endless source of inspiration.
%\end{acknowledgment}

%\begin{biog}
%\item[Woodrow Wilson] received his Ph.D. in history and political science from Johns Hopkins University. He held visiting positions at Cornell and %Wesleyan before joining the faculty at Princeton, where he was eventually appointed president of the university.  Among his proudest accomplishments %was the abolition of eating clubs at Princeton on the grounds that they were elitist.
%\begin{affil}
%Office of the President, Princeton University, Princeton NJ 08544\\
%twoodwilson@princeton.edu
%\end{affil}

%\item[Herbert Hoover] entered Stanford University in 1891, after failing all of the entrance exams except mathematics.  He received his B.S. degree in geology in 1895, spent time as a mining engineer, then was appointed by his co-author to the U.S. Food Administration and the Supreme Economic Council, where he orchestrated the greatest famine relief efforts of all time.
%\begin{affil}
%Hoover Institution, Stanford University, Stanford CA 94305\\
%herbhoover@stanford.edu
%\end{affil}
%\end{biog}
%\vfill\eject


\begin{thebibliography}{3}
\bibitem{aglermcnpfree} Jim Agler, John McCarthy, Pick Interpolation for free holomorphic functions \textit{Amer. Jour. Math} \textbf{137} (6) (2015) 1685--1701.
%\bibitem{hopkins} Brian Hopkins, ed., \textit{Resources for Teaching Discrete Mathematics}, Mathematical Association of America, Washington DC, 2009.
\bibitem{GAPPRIVATE} Thomas Breuer, The GAP~Group, \textit{private communication.}
\bibitem{GAP4} The GAP~Group, GAP -- Groups, Algorithms, and Programming, Version 4.10.0, 2018 (https://www.gap-system.org)
\bibitem{HARDY}Hardy, G. H. and Ramanujan, S. The Normal Number of Prime Factors of a Number n. \textit{Quart. J. Math.} \textbf{48}, 76-92, 1917.
\bibitem{GerstenhaberLAA} John Holbrook, K. C. O'Meara, Some thoughts on Gerstenhaber's theorem \textit{Lin. Alg. Appl.} \textbf{466} (2015) 267--295
\bibitem{GerstenhaberDec} K. C. O'Meara, The Gerstenhaber problem in characteristic zero is ``decidable" \textit{preprint}
\bibitem{knese} Greg Knese, Extreme points and saturated polynomials, \textit{preprint, March 2017} arXiv:1703.00094
\bibitem{Pap97} C.\, J. Pappacena, An upper bound for the length of a finite-dimensional algebra, {\em J. Algebra} {\bf 197} (1997), 535--545. 
\bibitem{Paz84}A. Paz, An application of the Cayley-Hamilton theorem to matrix polynomials in several variables, {\em Linear Multilinear Algebra} {\bf 15} (1984), 161--170.
\bibitem{Shitov} Yaroslav Shitov, An improved bound for the length of matrix algebras \textit{preprint} arXiv:1807.09310
\end{thebibliography}
\end{document}